\documentclass[12pt]{amsart}
\pdfoutput=1
\usepackage{amsmath, amsfonts, amssymb, amsthm}
\usepackage{enumerate}

\allowdisplaybreaks[1]
\usepackage{color}
\usepackage[pdftex]{graphicx}

\newcommand{\R}{\mathbb{R}}
\newcommand{\Z}{\mathbb{Z}}
\newcommand{\N}{\mathbb{N}}
\renewcommand{\P}{\mathbf{P}}
\newcommand{\B}{\mathcal{B}}

\newcommand{\ee}{\mathrm{e}}
\newcommand{\dx}{\mathrm{d}x}
\newcommand{\dy}{\mathrm{d}y}
\newcommand{\dd}{\mathrm{d}}

\newcommand{\Cv}{C_{\infty}}
\newcommand{\Cb}{C_{\delta}}
\newcommand{\Ca}{C_{a}}
\newcommand{\Crho}{C_{\rho}}

\newcommand{\ufl}{u_\mathrm{flat}}
\newcommand{\gfl}{g_\mathrm{flat}}
\newcommand{\usm}{u_\mathrm{smooth}}
\newcommand{\gsm}{g_\mathrm{smooth}}
\newcommand{\ustep}{u_\mathrm{step}}

\newcommand{\pv}{\mathrm{P.V.}}

\newcommand{\Laplace}{\Delta}
\newcommand{\halfLap}{-(-\Laplace)^{1/2}}
\newcommand{\A}{\operatorname{A}}
\newcommand{\BB}{\operatorname{B}}

\newcommand{\abs}[1]{\left| #1 \right|}
\newcommand{\norm}[1]{\left|\left| #1 \right|\right|}
\newcommand{\scalar}[1]{\left\langle  #1 \right\rangle}
\newcommand{\supp}{\operatorname{supp}}

\theoremstyle{plain}
\newtheorem{theorem}{Theorem}[section]
\newtheorem{lemma}[theorem]{Lemma}
\newtheorem{proposition}[theorem]{Proposition}
\newtheorem{corollary}[theorem]{Corollary}

\theoremstyle{definition}
\newtheorem{definition}[theorem]{Definition}
\newtheorem{assumption}[theorem]{Assumption}
\theoremstyle{remark}
\newtheorem{remark}[theorem]{Remark}

\title[Pinning in random elastic media]{Pinning of interfaces in a random elastic medium
and logarithmic lattice embeddings in percolation}
\date{\today}

\author{Patrick W. Dondl}
\address[P.~W.~Dondl]{Department of Mathematical Sciences\\
Durham University\\
Science Laboratories\\
South Rd\\
Durham DH1 3LE\\
United Kingdom}
\email{patrick.dondl@durham.ac.uk}
\urladdr{http://www.dondl.com/}

\author{Michael Scheutzow}
\address[M.~Scheutzow]{Fakult\"at II, Institut f\"ur Mathematik, Sekr.\ MA 7--5,
Technische Universit\"at Berlin, Strasse des 17.\ Juni 136,
D-10623 Berlin, Germany} \email{ms@math.tu-berlin.de}
\urladdr{http://www.math.tu-berlin.de/$\sim$scheutzow/}

\author{Sebastian Throm}
\address[S.~Throm]{Institut f\"ur Angewandte Mathematik\\
Im Neuenheimer Feld 294\\
D-69120 Heidelberg, Germany} \email{Throm@stud.uni-heidelberg.de}


\begin{document}

\begin{abstract} 
For a model of a driven interface in an elastic medium with random obstacles we prove existence of a stationary positive supersolution at non-vanishing driving force. This shows the emergence of a rate independent hysteresis through the interaction of the interface with the obstacles, despite a linear (force=velocity) microscopic kinetic relation. We also prove a percolation result, namely the possibility to embed the graph of an only logarithmically growing  function in a next-nearest neighbor site-percolation cluster at a non-trivial percolation threshold.
\end{abstract}
\maketitle
\section{Introduction and the main result}

In this article, we consider a model for the propagation of one-dimensional fronts immersed in an elastic medium subject to an external driving force and randomly distributed obstacles. The goal is to understand the overall macroscopic behavior of such fronts and its dependence on the external forcing. Here we prove existence of stationary solutions at positive driving force and thus the emergence of hysteresis.

In order to precisely state our model, let $(\Omega, \B, \P)$ be a probability space, $\omega \in \Omega$. The random front at time $t$ is given as the graph $(x,u(x,t,\omega))$ of a function $u\colon \R \times (0,\infty) \times \Omega \to \R$ solving the semilinear fractional diffusion problem
\begin{align}
u_t(x,t,\omega) &= \halfLap u(x,t,\omega) - f(x,u(x,t,\omega),\omega) + F  \label{eq:evolution}\\
u(x,0,\omega) &=0.  \nonumber
\end{align}
The function $f(x,y,\omega) \ge 0$ is assumed to be locally smooth in $x$ and $y$ for any $\omega$ and of the form of localized obstacles of identical shape and random positions with uniform density, i.e., the obstacle centers are given by a 2-dimensional Poisson process. See Assumption~\ref{ass:obstacles} for a precise statement. The constant term $F$ is an external loading, the fractional Laplacian models the interaction of the front with the elastic medium in which it is immersed.

\begin{figure}
\resizebox{0.8\textwidth}{!}{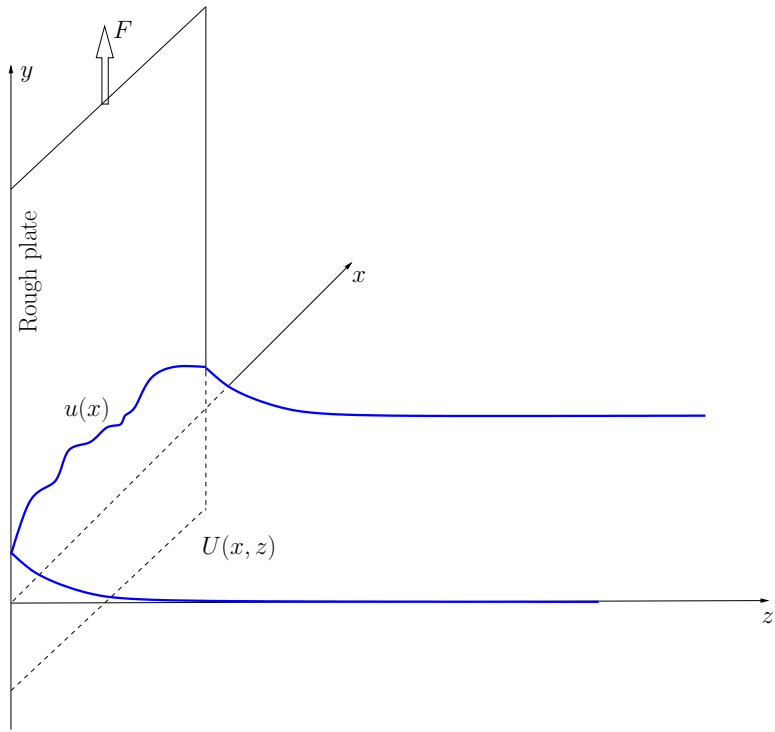}
\caption{Pulling sandpaper out of a glass of water.} \label{fig:sandpaper}
\end{figure}
Evolution problems of this kind arise in a large number of physical systems. A particularly simple example is that of ``pulling sandpaper out of a glass of water''. As illustrated in Figure~\ref{fig:sandpaper}, we model the evolution of the wetting line of the water surface on a rough plate as the plate gets pulled out of the water. Equation~\eqref{eq:evolution} can here formally be derived as follows, the derivation in other physical systems (e.g., crack fronts) is similar. We assume the motion of the wetting line $u\colon \R \to \R$ to be slow compared to the relaxation time of the water surface $U\colon \R \times \R^+ \to \R $. The system contains an energy term stemming from the water's surface energy, which is given (after removing the constant term from a completely flat surface) as
$$
\int_{\R \times \R^+} -1+\sqrt{ \abs{\nabla U}^2 +1 } \,\dx
$$
Linearizing this energy around a nearly flat state and cancelling the constant term we approximate this as
$$
\int_{\R \times \R^+} \frac{1}{2} \abs{\nabla{U}}^2 \,\dx.
$$
It is well known that the infimum of this energy subject to the condition that $U=u$ on the boundary of the domain is given by $\frac{1}{2}[u]^2_{H^{1/2}}$, i.e., the $H^{1/2}$-norm squared of the function setting the boundary condition. The variation of the $H^{1/2}$-norm squared yields the term containing the square root of the Laplacian in equation~\eqref{eq:evolution}. The constant term $F$ models the constant force with which the rough surface is pulled out of the water. The roughness of the surface itself acts as an obstacle to the evolution of the wetting line---it locally requires an additional amount of force to overcome a grain in the sandpaper. This is modelled by the heterogeneous force term $f$, yielding equation~\eqref{eq:evolution} as the viscous flow with respect to the involved force terms. An experimental example of this kind of system can be found in~\cite{Moulinet:2002jp}. They use a similar model, which is also proposed by~\cite{Ertas:1994, Joanny:552}. The non-local term in these articles is a mean-field version of our term.

Another important application in which a model of the above kind arises is that of a crack front propagating in a rough medium. Experiments and some modeling can be found in the work of Schmittbuhl et.~al.~\cite{Schmittbuhl:2003vc}. The derivation of the stress intensity factor for a non-flat crack front (resulting in the fractional Laplacian) was first given by Gao and Rice~\cite{Gao:1989}.  For simulations using the model and for more experimental references see for example~\cite{Tanguy:2004bs}.  Nonlocal operators that model the interaction with elastic media also arise in models for dislocations~\cite{Briani:2009uw, Forcadel:2009ec}.

In this article, we consider a specific form of the function $f$ which is that of localized smooth obstacles.\begin{assumption}[The random field]
\label{ass:obstacles}
Fix $\phi \colon \R^2 \to [0,1] \in C_c^\infty(\R^2)$ with 
\begin{enumerate}[\upshape (i)]
\item $\phi(x,y) = 0$ if $\norm{(x,y)} > r_1$,
\item $\phi(x,y) = 1$ if $\norm{(x,y)}_\mathrm{max} \le r_0$,
\end{enumerate}
for $r_1 > \sqrt{2}r_0 > 0$. By $\norm{\cdot}$ we denote here the Euclidean norm in $\R^2$, by $\norm{\cdot}_\mathrm{max}$ the maximum-norm in $\R^2$.
We consider $f$ to be of the form
$$
f(x,y,\omega) = \sum_{k\in K} f_k(\omega) \phi(x-x_k(\omega), y-y_k(\omega)).
$$
The random coefficients for the strength of obstacles $f_k$ are iid strictly positive random variables. The random distribution of obstacle sites $\{(x_k,y_k)\}_{k\in K}$ is a 2-dimensional Poisson process on $\R\times [r_1, \infty)$ with intensity $\lambda$.
\end{assumption}
\begin{remark}
Given Assumption~\ref{ass:obstacles} it is clear that~\eqref{eq:evolution} admits a unique classical solution.
\end{remark}

Under these conditions we can state the main theorem of this article.
\begin{theorem}[Pinning of interfaces]\label{thm:main}
Assume the function $f$ is chosen according to Assumption~\ref{ass:obstacles} and let $\A := -(-\Laplace)^s$ with $s \in [1/2,1)$. Then there exist a deterministic  $F^*>0$ and a $C^\infty$ random function $u\colon \R \times \Omega \to [0,\infty)$ such that
$$
0 \ge (\A u(\omega))(x) - f(x,u(x,\omega),\omega) + F^*
$$
for almost all $\omega\in\Omega$ and all $x\in \R$.
\end{theorem}
\begin{remark}
By the comparison principle, of course any random field that can be bounded from below for a.e.~$\omega$ by a field of the type of Assumption~\ref{ass:obstacles} yields the same pinning result.
\end{remark}

This theorem states that there is a non-trivial pinning threshold in our model, since by the comparison principle any solution of the fractional diffusion problem~\eqref{eq:evolution} with $F\le F^*$ and zero initial condition must remain below the non-negative supersolution $u$. 
Note that Assumption~\ref{ass:obstacles} ensures that an identically zero function is a stationary subsolution to the evolution problem~\eqref{eq:evolution} for any $F\ge 0$. Thus, for $0\le F \le F^*$, the interface becomes trapped and reaches---at least asymptotically---a stationary state.

The article is organized as follows. In Section~\ref{sec:flat} we show the existence of a non-trivial  threshold for the existence of infinite percolation clusters which contain the graph of a function that only grows logarithmically. This is a generalization of Lipschitz percolation~\cite{Dirr:2010}, the proof for our result is inspired by~\cite{Grimmett:2010vc}. In Section~\ref{sec:construction}, using the percolation result, the supersolution is constructed. In contrast to~\cite{Dirr:2011}, due to the non-local nature of the problem, a simple piecewise construction is no longer sufficient. Finally, in Section~\ref{sec:conclusions} we present some conclusions and open problems.

\section{Flat Percolation Clusters}
\label{sec:flat}

In this section, let $\|.\|$ denote the $l_1-$norm on $\R^n$ and denote the $i^{\mathrm{th}}$ unit vector in $\R^m$ by $e_i$ ($i \le m$).
\begin{theorem} 
\label{thm:perc}
Consider {\em site percolation} on $\Z^{n+1}$ with $n \ge 1$: for given $p \in [0,1]$, each $z \in \Z^{n+1}$ is 
{\em open} with probability $p$ and {\em closed} otherwise, with different sites receiving independent states. 
For each nondecreasing function $H:\N_0 \to \N_0$ satisfying 
\begin{itemize}
\item[i)] $H(0)=0$
\item[ii)] $H(1)\ge 1$
\item[iii)] $\liminf_{k \to \infty} \frac{H(k)}{\log k} >0$,
\end{itemize}
there exists some $p_H=p_H(n) \in (0,1)$ such for each $p \in (p_H,1)$ and almost every realization of the site percolation model 
with parameter $p$, there exists a (random) function $\Lambda:\Z^n \to \N$ which satisfies
\begin{itemize}
\item[a)] $|\Lambda(x)-\Lambda(y)| \le H(\|x-y\|)$ for all $x,y \in \Z^n$
\item[b)] $(x,\Lambda(x))$ is open for every $x \in \Z^n$.
\end{itemize} 
\end{theorem} 
\begin{proof}
It suffices to prove the theorem in case $H(k+1) \le H(k) +1$ for all $k \in \N_0$ (which implies $H(1)=1$). 
Further, we can and will assume without loss of generality that for all positive integers $k_1,...,k_m$, we have $H(\sum_{j}k_j)\le \sum_j H(k_j)$. 

For $j \in \N$ define
$$
R(j):=\sup\{k \in \N: H(k)=j\}.
$$
Now we explain what we mean by an admissible path. Let $x,y \in \Z^{n+1}$. A {\em blocking-path} from $x$ to $y$ is a finite 
sequence of distinct sites $x=x_0,x_1,...,x_k=y$ in $\Z^{n+1}$ such that for each $i=1,...,k$ the difference $x_i-x_{i-1}$ 
takes either the value $e_{n+1}$ or  $(z,-H(\|z\|))$ for some $z \in \Z^n\backslash \{0\}$. 
A {\em blocking-path} is called an {\em admissible path} if in addition for each $i=1,...,k$ we have that
if $x_i-x_{i-1}=e_{n+1}$ then $x_i$ is closed.
For $z \in \Z^n$, define
\begin{gather*}
\Lambda(z):=1+\sup\{h \in \N_0: \\ \mbox{ there exist $x \in \Z^n$ and an admissible path from $(x,0)$ to $(z,h)$}\}.
\end{gather*}
If this supremum is finite for some $z$, then $(z,\Lambda(z))$ is clearly open. In this case, $\Lambda(x)$ is finite for all $x \in \Z^n$ and $\Lambda$ satisfies a) in the 
Theorem. All that remains to be shown is that $\Lambda(0)$ is finite. 

By assumption, there exist $C>0$ and $\gamma >0$ such that $R(i) \le C \ee^{\gamma i}$ for all $i \in \N$. 
Observe, that there exists some $K$ such that the number of sites in $\Z^n$ with $l_1-$norm at most $k$ is bounded by $K k^n$ for all 
$k \in \N$.
Fix $N \in \N_0$, $h \in \N$, and $q:=1-p \in (0,1)$. For  $x \in \Z^n$ with $\|x\|=N$ we estimate the expected number of admissible paths from $(x,0)$ 
to $(0,h)$ as follows:

for a given such admissible path, let $k_i$ be the number of steps of the path containing a down-jump of size $i \in \N$. Then 
the expected number of such admissible paths which contain exactly $D \in \N_0$ down-steps (in the sense that 
$\sum_{i=1}^k (x_i-x_{i-1})^-=D$) and therefore $D+h$ up-steps is at most
\begin{align*}
\sum &\Big( {{k_1+...+k_D+D+h} \choose {k_1,...,k_D,D+h}} q^{D+h}\prod_{i=1}^D(K\, R(i)^n)^{k_i}\Big)\\
&\le q^{D+h} \ee^{\gamma nD} \sum\Big( {{k_1+...+k_D+D+h} \choose {k_1,...,k_D,D+h}}  (KC^n)^{k_1+...+k_D}\Big)\\
&\le q^{D+h} \ee^{\gamma nD} \sum_{m=0}^D \Big((KC^n)^m \sum {{m+D+h} \choose {k_1,...,k_D,D+h}} \Big)\\
&= q^{D+h} \ee^{\gamma nD} \sum_{m=0}^D \Big((KC^n)^m {{m+D+h}\choose {D+h}}\sum {m \choose {k_1,...,k_D}} \Big)\\
&\le q^{D+h} \ee^{\gamma nD} \sum_{m=0}^D \Big((KC^n)^m 2^{m+D+h} 2^D {{m+D-1}\choose {m}}\Big)\\
&\le q^{D+h} \ee^{\gamma nD} 2^{2D+h} ((2KC^n)\vee 1)^D \sum_{m=0}^D {{m+D-1}\choose {m}}\\
&= q^{D+h} \ee^{\gamma nD} 2^{2D+h} ((2KC^n)\vee 1)^D {{2D}\choose {D}}\\
&\le q^{D+h} \ee^{\gamma nD} 2^{2D+h} ((2KC^n)\vee 1)^D 2^{2D},
\end{align*} 
where the first two sums are extended over all $k_1,...,k_D \in \N_0$ satisfying $\sum_{i=1}^D i k_i=D$ and the fourth and sixth sums   
extend over all $k_1,...,k_D \in \N_0$ which in addition satisfy  $\sum_{i=1}^D k_i=m$. Let
$$
\beta:=16 \ee^{\gamma n}((2KC^n)\vee 1).
$$
Summing over $D$ from $H(N)$ to $\infty$, we see that, for $q\beta<1$, the expected number of admissible paths from $(x,0)$ to $(0,h)$ is at most 
$(2q)^h (q\beta)^{H(N)} (1-q\beta)^{-1}$. The total expected number of admissible paths starting from any point $(x,0),\;x \in \Z^n$ and ending at $(0,h)$ is bounded by 
$$
(2q)^h(1-q\beta)^{-1}\sum_{N=0}^{\infty} \Big(\big( \big( \tilde K N^{n-1}\big)\vee 1\big) (q\beta)^{H(N)}\Big).
$$ 
Here, $\tilde K$ is a constant chosen such that the number of $x \in \Z^n$ such that $\|x\|=N$ is bounded by $\tilde K N^{n-1} $ for all $N \in \N$. 
The sum is clearly finite provided $q>0$ is sufficiently small. Now, the first Borel-Cantelli Lemma implies that the largest $h$ for which there exists an admissible 
path from some $(x,0)$ and ending at $(0,h)$ is finite almost surely and therefore the assertion is proved.
\end{proof}

\begin{remark}
The theorem is sharp in the sense that it becomes wrong if in iii) ``$\inf$'' is dropped and ``$>$'' is replaced by ``$=$'' (this is a consequence of the second Borel-Cantelli Lemma) 
\end{remark}

\section{Construction of the supersolution}
\label{sec:construction}
The construction of the supersolution is performed in a series of steps. We first split up $\R^2$ into boxes large enough so that boxes that contain an obstacle of a minimum strength percolate in the sense of Section~\ref{sec:flat}. All obstacles not necessary for the percolation cluster are then disregarded. In each column of boxes we now have one  obstacle at position $(x_i, y(x_i))$. Starting from a periodic supersolution (assuming obstacles at $y=0$ and at periodic distance in $x$ with period larger than the box size), we construct a supersolution for obstacles centered at $(x_i, 0)$ by cutting out one period and using this function locally around obstacle sites. Finally, we can add a smooth function with less-than-linear growth (given by the percolation cluster) in order to obtain a supersolution that passes through the original obstacle sites.

In this section, we make frequent use of the equivalence of the integral representation and the Fourier representation of the fractional Laplacian. Furthermore, we use the symmetry of the fractional Laplace operator and the weak form of it by switching between applying it to a test function and the function itself. Further information can be found in~\cite{DiNezza:2011va}. The extension problem related to fractional Laplacians has been treated in~\cite{Caffarelli:2007en}

\begin{definition}
\label{def:g}
Consider thus first $a, b, \delta, F_2 >0$ with $a>4b$ and $\delta < \min\{1,b\}$. Let $\rho := b+\delta/2$. Let $F_1=\frac{\rho}{a-\rho}F_2$. We define
$$
\tilde{g}(x) := \left\{ \begin{array}{ll} F_2 &\quad \textrm{for $x\in [-b-\delta/2, b+\delta/2] = [-\rho, \rho]$} \\
-F_1 &\quad \textrm{for $x\in[-a,a]\setminus [-\rho, \rho]$},
\end{array} \right.
$$
periodically extended to the real line. Now let 
$$
g := \eta_{\delta/2} * \tilde{g},
$$
with $\eta_{\delta/2}$ a standard-mollifier\footnote{It is important that this standard mollifier used here is symmetric with respect to the origin.} with $|\supp \eta_{\delta/2} | = \delta$, i.e., radius $\delta/2$.
\end{definition}
\begin{remark}
Note that the following statements hold.
\begin{enumerate}[\upshape (i)]
\item $g$ is periodic with period $2a$,
\item $g(x) = F_2$ for $x\in (-b,b)$,
\item $g(x) = -F_1$ for $x\in (-2a+b+\delta, 2a-b-\delta) \setminus [-b-\delta, b+\delta]$
\item $\tilde{g}$ and $g$  have vanishing averages.
\end{enumerate}
\end{remark}

\begin{definition}
\label{def:v}
In order to construct a periodic supersolution we let $\tilde{v}$ be the modulo a constant unique and continuous periodic solution of
$$
\A\tilde{v} = \tilde{g}
$$
and set $v = \eta_{\delta/2} * \tilde{v}$.
The constant is set so that the average of $v$ vanishes.
\end{definition}
\begin{remark}
Note that $\A v = g$.
\end{remark}

An explicit uniformly converging Fourier series representation of $\tilde{v}$ can now easily be obtained and yields $L^\infty$-bounds on $\tilde{v}$ and $v$ as well as some symmetry properties.
\begin{lemma}\label{lem:stand_sol}
We have
\begin{enumerate}[\upshape (i)]
\item $\tilde{v}\left(x\right)=-2a^{2s}\left(\sum_{k=1}^{\infty}\frac{F_1+F_2}{\pi^{1+2s}}\sin\left(k\frac{\pi}{a}(b+\delta/2)\right)\frac{\cos\left(k\frac{\pi}{a}x\right)}{k^{1+2s}}\right) $
\item  $v(-x)=v(x)$ for all $x\in \mathbb R$, $\tilde{v}(-x)=\tilde{v}(x)$ for all $x\in \mathbb R$, $v$ and $\tilde{v}$ are periodic with period $2a$ and $v$ is continuous.

\item $\left\|v\right\|_{\infty} \leq \left\|\tilde{v}\right\|_{\infty}\leq\frac{2\left(F_1+F_2\right)}{\pi^{2s}}\zeta (2s)a^{2s-1}\left(b+\delta/2\right)$ for $s>1/2$

\item $\left\|v\right\|_{\infty} \leq\left\|\tilde{v}\right\|_{\infty} \\ \leq \frac{2}{\pi}\left(F_1+F_2\right)\left(b+\delta/2\right)\left(2+\log\left(a\right)-\log\left(\pi\left(b+\delta/2\right)\right)\right)$ for $s=1/2$
 
\end{enumerate}
Here, $\zeta$ denotes the Riemann zeta function.
\end{lemma}
\begin{proof}
The series representation follows from a straight-forward calculation, properties (ii)--(iii) follow directly from the representation by bounding the absolute value of the sine and cosine functions by unity. Property (iv) can be seen as follows. For $s=1/2$ we have
\begin{align*} \abs{\tilde{v}(x)}&=\abs{-2a\sum_{k=1}^{\infty}\frac{F_1+F_2}{\pi^2}\sin\left(k\frac{\pi}{a}\left(b+\delta/2\right)\right)\frac{\cos\left(k\frac{\pi}{a}x\right)}{k^2}}\\
&=\frac{2\left(F_1+F_2\right)}{\pi^2}a\abs{\sum_{k=1}^{\infty}\frac{\sin\left(k\frac{\pi}{a}\left(b+\delta/2\right)\right)\cos\left(k\frac{\pi}{a}x\right)}{k^2}}.
\end{align*}
With $\abs{\sin\left(k\frac{\pi}{a}\left(b+\delta/2\right)\right)}\leq k\frac{\pi}{a}(b+\delta/2)$ and $\abs{\cos\left(k\frac{\pi}{a}x\right)} \leq 1$ and $\abs{\sin\left(k\frac{\pi}{a}\left(b+\delta/2\right)\right)}\leq 1$ by splitting up the sum in two parts and using an integral estimate for each part one gets

\begin{align*}
\quad\abs{\tilde{v}(x)} \leq& \frac{2\left(F_1+F_2\right)}{\pi^2}a\left(\frac{\pi}{a}\left(b+\delta/2\right)+\int_{1}^{\frac{a}{\pi\left(b+\delta/2\right)}}\frac{\pi}{a}\left(b+\delta/2\right)\frac{1}{k}\, \mathrm{d} k+ \right. \\
& \left.\int_{\frac{a}{\pi\left(b+\delta/2\right)}}^{\infty}\frac{1}{k^2}\, \mathrm{d} k\right)\\
=&\frac{2\left(F_1+F_2\right)}{\pi^2}\left(\pi\left(b+\delta/2\right)\left(1+\log\left(\frac{a}{\pi\left(b+\delta/2\right)}\right)\right)+\frac{a}{\frac{a}{\pi\left(b+\delta/2\right)}}\right)\\
=&\frac{2\left(F_1+F_2\right)}{\pi}\left(b+\delta/2\right)\left(2+\log\left(a\right)-\log\left(\pi\left(b+\delta/2\right)\right)\right).
\end{align*}
\end{proof}

It is now necessary to establish some monotonicity properties of the function $v$.
\begin{lemma}\label{lem:MonVt}
The function $\tilde{v}$ strictly increases on $[0, a]$ (and thus by symmetry strictly decreases on $[-a, 0]$).
\end{lemma}
We prove this lemma as a consequence of the following two Propositions by showing positivity (respectively, negativity) of $\tilde{v}'$. This property remains valid under mollification, as shown in Lemma~\ref{lem:monoticity}.

We denote by $E^+ := \bigcup_{k=-\infty}^{\infty}\left(2ka-b-\delta/2,\, 2ka+b+\delta/2\right)$ the set where the second derivative $\tilde{v}''$ will be shown to be strictly positive and $E^- := \R\setminus \overline{E^+}$ the set where the second derivative will be shown to be strictly negative. We also first have to show smoothness of $\tilde{v}$ on the union of those two sets. 
\begin{proposition}
\label{prop:mon_gt}
Let $p\in (0,1)$. We have that $\left(-\Delta\right)^{p} \tilde{g}\left(x\right)$ (given by its integral representation) exists for all $x\in E^+\cup E^-$ and one has
$$
	\left(-\Delta\right)^{p}\tilde{g}\left(x\right)
	\begin{cases}
	>0 & x\in E^+\\
	<0 & x\in E^-
	\end{cases}
$$
\end{proposition}
\begin{proof}
Take $x\in E^+$, then one has (with some $C>0$)
\begin{align*}
	\left(-\Delta\right)^{p}\tilde{g}\left(x\right)=& C\, \pv \int_{\mathbb R}\frac{\tilde{g}\left(x\right)-\tilde{g}\left(y\right)}{\abs{x-y}^{1+2p}}\, \dy \\
	=& C\, \pv\int_{\mathbb R}\frac{F_1+F_2}{\abs{x-y}^{1+2p}}\chi_{E^+}\left(y\right) \, \dy >0
\end{align*}
where $\chi_{E^+}$ is the characteristic function of $E^+$. For $x\in E^-$ the same calculation gives
\begin{align*}
	\left(-\Delta\right)^p\tilde{g}\left(x\right)=&C\, \pv \int_{\mathbb R}\frac{\tilde{g}\left(x\right)-\tilde{g}\left(y\right)}{\abs{x-y}^{1+2p}}\, \dy \\
	=&-C\, \pv\int_{\mathbb R}\frac{F_1+F_2}{\abs{x-y}^{1+2p}}\chi_{E^-}\left(y\right) \, \dy <0
\end{align*}
\end{proof}

\begin{proposition}
\label{prop:vpp}
For all $x\in  E^+ \cup E^-$ the function $\tilde{v}$ is twice differentiable with 
\begin{align*}
\tilde{v}''\left(x\right)=\Delta \tilde{v} \left(x\right)=\left(-\Delta\right)^{1-s}\left(-\left(-\Delta\right)^s\tilde{v}\left(x\right)\right)=-\left(-\Delta\right)^{1-s}\tilde{g}\left(x\right).
\end{align*}
\end{proposition}

\begin{proof}
Let $\BB:=(-\Delta)^{1-s}$

For $\BB$ given by the integral representation $\BB \tilde{g}$ is continous on $(-\rho,\, \rho)$ and $\left[-a,\, a\right]\backslash \left[-\rho,\, \rho\right]$.
For $\left[-\beta,\, \beta \right]\subset (-\rho,\, \rho)$ take $\epsilon>0$ so small that $\left[-\beta-2\epsilon,\, \beta+2\epsilon\right]\subset (-\rho,\, \rho)$.

Take a standard mollifier $\eta_{\epsilon}$ with radius $\epsilon$. Define $\tilde{v}_{\epsilon}:=\tilde{v}\ast\eta_{\epsilon}$ and take a test function $\psi\in C_c^{\infty}$ with support $\supp \psi \subset \left[-\beta,\, \beta\right]$. Then one has for $x\in\left[-\beta,\, \beta\right]$
\begin{align*}
&\quad \tilde{v}_{\epsilon}''(x)=\Delta \tilde{v}_{\epsilon}(x)=-\BB(\A \tilde{v}_{\epsilon}(x))=-\BB(\A(\tilde{v}\ast\eta_{\epsilon}(x)))\\
&=-\BB((\A\tilde{v})\ast\eta_{\epsilon}(x))=-\BB(\tilde{g}\ast\eta_{\epsilon}(x))
\end{align*}
As $\tilde{g}\ast\eta_{\epsilon}$ is smooth and bounded one has 
\begin{align*}
&\quad \scalar{-\BB(\tilde{g}\ast\eta_{\epsilon}),\, \psi}=\scalar{\tilde{g}\ast \eta_{\epsilon},\, -\BB\psi}=\scalar{\tilde{g},\, -(\BB\psi)\ast\eta_{\epsilon}}\\
&= \scalar{\tilde{g},\,-\BB(\psi\ast\eta_{\epsilon}) }
\end{align*}
Using this together with the integral representation of $\BB$ one gets up to some constant
\begin{align*}
&\quad2\scalar{\tilde{g},\, \BB(\psi\ast\eta_{\epsilon})}\\
&=\int_{\R}\tilde{g}(x)\left(\int_{\R}\frac{\psi\ast\eta_{\epsilon}(x)-\psi\ast\eta_{\epsilon}(y)}{\abs{x-y}^{1+2(1-s)}}\dy\right)\dx \\
&\quad -\int_{\R}\tilde{g}(y)\left(\int_{\R}\frac{\psi\ast\eta_{\epsilon}(x)-\psi\ast\eta_{\epsilon}(y)}{\abs{x-y}^{1+2(1-s)}}\dx\right)\dy\\
&=\int_{\R}\int_{\R}\frac{\tilde{g}(x)-\tilde{g}(y)}{\abs{x-y}^{3-2s}}(\psi\ast\eta_{\epsilon}(x)-\psi\ast\eta_{\epsilon}(y))\dx\dy\\
&=\int_{\R}\int_{\R}\frac{\tilde{g}(x)-\tilde{g}(y)}{\abs{x-y}^{3-2s}}\psi\ast\eta_{\epsilon}(x)\dy\dx-\int_{\R}\int_{\R}\frac{\tilde{g}(x)-\tilde{g}(y)}{\abs{x-y}^{3-2s}}\psi\ast\eta_{\epsilon}(y)\dx\dy\\
&=\int_{-\beta-\epsilon}^{\beta+\epsilon}\left(\int_{\R}\frac{\tilde{g}(x)-\tilde{g}(y)}{\abs{x-y}^{3-2s}}\dy\right)\psi\ast\eta_{\epsilon}(x)\dx\\
&\quad -\int_{-\beta-\epsilon}^{\beta+\epsilon}\left(\int_{\R}\frac{\tilde{g}(x)-\tilde{g}(y)}{\abs{x-y}^{3-2s}}\dx\right)\psi\ast\eta_{\epsilon}(y)\dy\\
&=2\scalar{\BB \tilde{g},\, \psi\ast\eta_{\epsilon}}
\end{align*}
From the choice of $\left[-\beta,\, \beta\right]$, $\epsilon$ as well as $\supp \psi\subset \left[-\beta,\, \beta\right]$ one has
\begin{align*}
\quad\scalar{\BB \tilde{g},\, \psi\ast\eta_{\epsilon}}&=\int_{-\beta-\epsilon}^{\beta+\epsilon}\BB \tilde{g}(x)\int_{\-\beta}^{\beta}\psi(y)\eta_{\epsilon}(x-y)\dy\dx\\
&=\int_{-\beta-\epsilon}^{\beta+\epsilon}\int_{-\beta}^{\beta}\BB \tilde{g}(x)\psi(y)\eta_{\epsilon}(x-y)\dy\dx\\
&=\int_{-\beta-\epsilon}^{\beta+\epsilon}\int_{-\beta}^{\beta}\BB \tilde{g}(x)\psi(y)\eta_{\epsilon}(y-x)\dy\dx\\
&=\int_{-\beta}^{\beta}\psi(y)\int_{-\beta-\epsilon}^{\beta+\epsilon}\BB \tilde{g}(x)\eta_{\epsilon}(y-x)\dx\dy=\scalar{\BB \tilde{g}\ast\eta_{\epsilon},\, \psi}
\end{align*}
Finally this shows that
\begin{align*}
\tilde{v}_{\epsilon}''(x)=-\BB(\tilde{g}\ast\eta_{\epsilon}(x))=(-\BB \tilde{g})\ast \eta_{\epsilon}(x) 
\end{align*}
for all $x\in\left[-\beta,\, \beta\right]$.

In the same way one can show that
\begin{align*}
\tilde{v}_{\epsilon}''(x)=-\BB(\tilde{g}\ast\eta_{\epsilon}(x))=(-\BB \tilde{g})\ast \eta_{\epsilon}(x) 
\end{align*}
for all $x\in\left[-\beta,\, \beta\right]\backslash \left[-\rho,\, \rho\right]$ with $\left[-\beta,\, \beta\right]\backslash \left[-\rho,\, \rho\right]\subset \left[-a,\, a\right]\backslash \left[-\rho,\, \rho\right]$

From the continuity of $-\BB \tilde{g}$ on $\left[-\beta,\, \beta\right]$ one gets that $\tilde{v}_{\epsilon}''=(-\BB \tilde{g})\ast \eta_{\epsilon}$ converges uniformly to $-\BB \tilde{g}$ on $\left[-\beta,\, \beta\right]$. 

Furthermore one has

\begin{enumerate}[\upshape (i)]
	\item $\tilde{v}_{\epsilon}\rightarrow \tilde{v}$ uniformly on $\left[-\beta,\, \beta\right]$
	\item $\tilde{v}\in H^1(\left[-\beta,\,\beta\right])$ and therefore $\tilde{v}''=\Delta \tilde{v}\in H^{-1}\left(\left[-\beta,\,\beta\right]\right)$
\end{enumerate}

Then one has for any test function $\psi\in C_c^{\infty}$ with support in $\left[-\beta,\, \beta\right]$
\begin{align*}
\scalar{\tilde{v},\, \psi''}=\lim_{\epsilon\to 0}\scalar{\tilde{v}_{\epsilon},\, \psi''}=\lim_{\epsilon\to 0}\scalar{\tilde{v}_{\epsilon}'',\, \psi}=\scalar{-\BB \tilde{g},\, \psi}
\end{align*}

This shows that $\tilde{v}''= -\BB \tilde{g}$ in $H^{-1}(\left[-\beta,\, \beta\right])$ and in the sense of distributions. As $-\BB \tilde{g}$ is continuous on $\left[-\beta,\, \beta\right]$ this proves that $\tilde{v}\in C^2\left(\left[-\beta,\,\beta\right]\right)$.

As $\left[-\beta,\, \beta\right]\subset \left(-\rho,\, \rho\right)$ was arbitrary one gets

\begin{align*}
\tilde{v}''(x)=-\BB\tilde{g}(x)=-(-\Delta)^{1-s}\tilde{g}(x)
\end{align*}
for all $x\in \left(-\rho,\, \rho\right)$.

In the same way one shows the assertion for $x\in \left[-a,\,a\right]\backslash \left[-\rho,\, \rho\right]$ and from periodicity the statement follows.
\end{proof}

\begin{remark}
For symmetry reasons we have $\tilde{v}'(0)=\tilde{v}'(a)=0$.
\end{remark}

\begin{proof}[Proof of Lemma~\ref{lem:MonVt}]
According to Proposition~\ref{prop:vpp}, $\tilde{v}''$ exists on $E^+\cup E^-$ and is strictly positive on $E^+$ and strictly negative on $E^-$. Because of $\tilde{v}'\left(0\right)=0=\tilde{v}'\left(a\right)$ we get for $x\in \left[0,\, a\right]\backslash\left\{b+\delta/2\right\}$ by the fundamental theorem of calculus
\begin{align*}
\tilde{v}'\left(x\right)&=
\begin{cases}
\int_{0}^{x}\left(-\Delta\right)^{1-s}\tilde{g}\left(y\right)\dy & x\in \left[0,\, b+\delta/2\right)\\
-\int_{x}^{a}\left(-\Delta\right)^{1-s}\tilde{g}\left(y\right)\dy & x\in \left(b+\delta/2,\, a\right]
\end{cases}\\
&>0
\end{align*}
where in the last step Proposition~\ref{prop:mon_gt} was used.
\end{proof}

\begin{lemma}\label{lem:monoticity}
The function $v$ strictly increases on $\left[0,\, a\right]$ (and therefore by symmetry strictly decreases on $\left[-a,\, 0\right]$).
\end{lemma}
\begin{proof}
Note first that $\tilde{v}' \in L^2_{loc}$ as a weak derivative.
Since 
$$
	v'=\frac{\dd}{\dx}\left(\tilde{v} * \eta_\delta \right)\equiv \left(\frac{\dd}{\dx}\tilde{v}\right)* \eta_\delta
$$
one gets
$$
	v'\left(x\right)=\left(\left(\frac{\dd}{\dx}\tilde{v}\right)* \eta_\delta \right)\left(x\right)>0
$$
for all $x\in \left(0,\, a\right)$ according to Lemma~\ref{lem:MonVt} and because the convolution preserves positivity in this case due to the symmetry properties of $\tilde{v}$ and the mollifier.

\end{proof}

We now split $\R^2$ into boxes large enough so that the percolation theorem from Section~\ref{sec:flat} can be applied to boxes that contain an obstacle.

\begin{definition}
For $k \in \Z$, $j \in \N$ and $l,d,h >0$, $l>2r_1$ let
\begin{align*}
Q_k :=& [k(l+d) - l/2, k(l+d)+l/2] \\ 
\tilde{Q}_k :=& [k(l+d) - l/2 + r_1, k(l+d)+l/2 - r_1] \\
\tilde{Q}_{kj} :=& \tilde{Q}_k \times  [ (j-1)h, jh]
\end{align*}
\end{definition}

The following is a direct result from Section~\ref{sec:flat} and Assumption~\ref{ass:obstacles}.
\begin{proposition}\label{prop:prop}
Let $0<\alpha<1$, $H(k) := \lfloor k^\alpha \rfloor$ (i.e., the integer floor of $k^\alpha$) and let $V := (l-2r_1)h >0$, $q>0$ such that 
$$
1- \exp\{-\lambda V \cdot \P\{f_1 \ge q\}  \} > p_H(1)
$$
from Theorem~\ref{thm:perc}. Then, almost surely, there exist a random function $\Lambda \colon \Z \to \N$ with $\abs{\Lambda(x)-\Lambda(y)} \le H(\abs{x-y})$ for all $x,y\in\Z$ and a mapping $I \colon \Z \to K$ with 
$$
(x_{I(k)}, y_{I(k)}) \in \tilde{Q}_{k,\Lambda(k)}, \quad f_{I(k)}\ge q
$$
for all $k\in \Z$. In the following we denote by $I$ the set $I(\Z)$.
\end{proposition}

\begin{definition}
\label{def:ufl}
Let now $d\ge l$, $2a \ge d+2l$ and define for $i\in I$
\begin{align*}
u_i(x) &:= \left\{ \begin{array}{ll} v(x-x_i) & \textrm{for $x\in [x_i-l-d/2,x_i+l+d/2]$} \\
+\infty & \textrm{otherwise.} \end{array}\right. \\
v_i(x) &:= v(x-x_i) \\
\ufl(x) &:= \min_{i\in I} v_i(x),
\end{align*}
where the $2a$-periodic function $v\colon \R \to \R$ is given in Definition~\ref{def:v}.
\end{definition}

\begin{remark}\label{rem:intersection}
Note that due to the monotonicity and periodicity properties\footnote{The function $v$ admits one maximum and one minimum in each period and it is strictly monotone in between.} of $v$ from Lemma~\ref{lem:monoticity}, we have that on any interval of the form $(b, b+2a]$ two functions $v_i$, $v_j$, $i,j\in I$ intersect each other exactly twice or they are identical. The points of intersection have distance $a$.
\end{remark}

\begin{proposition}\label{prop:welldefined}
With the definitions above we have that
\begin{enumerate}[\upshape (i)]
\item $\ufl$ is bounded and continuous,
\item given $x_i \in Q_k$, we have that $\ufl(x) = v_i(x)$ for all $x\in Q_k$.
\end{enumerate}
\end{proposition}
\begin{proof}
(i) is obvious by construction. In order to see (ii), note that the spacing $d$ between two boxes $Q_j$ is larger than the length $l$ of a box and the period of the function $v$ is larger than $d/2+l$.
\end{proof}

In the following, we prove that the function $\ufl$ constructed above is a supersolution to a modified problem where the  obstacles are extended from $-\infty$ to $+\infty$ in the $y$-direction. We thus fix $\xi \in \R$ and calculate the effect of the fractional Laplacian $\A$ on $\ufl$ evaluated at $\xi$. In the case that $\xi$ is a point where $\ufl$ is smooth, i.e., not a point where the minimizing $v_i$ in Definition~\ref{def:ufl} changes, we can directly apply the integral representation of $\A$. The points of discontinuity (of the first derivative) of $\ufl$ will have to be smoothed in order to construct a $C^\infty$ supersolution.

\begin{definition}
Given $\xi \in \R$, let $i_0 \in I$ such that $\ufl(\xi) = v_{i_0}(\xi) = u_{i_0}(\xi)$. If $v_j(\xi) = v_i(\xi) = \ufl(\xi)$ we take $i_0 = \max \{i,j\}$. Furthermore, we recursively define the points of intersection of the periodic supersolution $u_{i_0}$ with $\ufl$. Let
\begin{align*}
a_1 :=& \min\left\{ y\ge 0 : \exists \kappa >0 
	\;\textrm{with}\; \begin{array}{ll} u_{i_0} > \ufl & \textrm{on $(\xi-y-\kappa, \xi-y)$} \\ 
					                u_{i_0} \le \ufl & \textrm{on $(\xi-y, \xi-y+\kappa)$}  \end{array}    \right\}, \\
b_1 :=& \min\left\{ y\ge a_1 : \exists \kappa >0 
	\;\textrm{with}\; \begin{array}{ll} u_{i_0} < \ufl & \textrm{on $(\xi-y-\kappa, \xi-y)$} \\ 
					                u_{i_0} > \ufl & \textrm{on $(\xi-y, \xi-y+\kappa)$}  \end{array}    \right\}, \\		    
a_{k+1} :=& \min\left\{ y\ge b_k : \exists \kappa >0 
	\;\textrm{with}\; \begin{array}{ll} u_{i_0} > \ufl & \textrm{on $(\xi-y-\kappa, \xi-y)$} \\ 
					                u_{i_0} < \ufl & \textrm{on $(\xi-y, \xi-y+\kappa)$}  \end{array}    \right\}, \\
b_{k+1} :=& \min\left\{ y\ge a_{k+1} : \exists \kappa >0 
	\;\textrm{with}\; \begin{array}{ll} u_{i_0} < \ufl & \textrm{on $(\xi-y-\kappa, \xi-y)$} \\ 
					                u_{i_0} > \ufl & \textrm{on $(\xi-y, \xi-y+\kappa)$}  \end{array}    \right\}.
\end{align*}

Define $\tilde{a}_k, \tilde{b}_k$ in the same way by substituting $-y$ by $+y$

\end{definition}

\begin{lemma}\label{lem:distance}
We have
$$
b_k - a_k \ge a \quad \textrm{and} \quad a_{k+1} - b_k \le a
$$
for all $k\in\N$.
\end{lemma}
\begin{proof}
Take $k\in \mathbb N$ arbitarily and $i_1,i_2\in I$ such that $v_{i_1}(\xi-a_k)=\ufl(\xi-a_k)$, $v_{i_2}(\xi-b_k)=\ufl(\xi-b_k)$ (it is clear by constrcution that $i_1$ and $i_2$ are unique).

Because of Remark~\ref{rem:intersection} and the construction of $\ufl$ one has
\begin{align*}
 \ufl &\leq u_{i_1} \text{ on } \left[\xi-a_k-a,\xi-a_k\right]. 
\end{align*}
The definition of $a_k$ and Remark~\ref{rem:intersection} yield
	\[
	u_{i_0}(\xi-a_k)=u_{i_1}(\xi-a_k) \text{ and } u_{i_0}(x)>u_{i_1}(x) \text{ for } x\in (\xi-a_k-a,\xi-a_k).
\]

Altogether one has
	\[
	u_{i_0}(x)>u_{i_1}(x)\geq \ufl(x) \quad \text{for all}\quad x\in (\xi-a_k-a,\xi-a_k).
\]
Furthermore
	\[
	\ufl\leq u_{i_0} \text{ on } \left[\xi-b_k,\xi-a_k\right]
\]
and by the choice of $a_k$ and $b_k$ there exists no larger interval $J\supseteq\left[\xi-b_k,\xi-a_k\right]$ with $\ufl\leq u_{i_0}$ on $J$. Using this one gets
	\[
	(\xi-a_k-a,\xi-a_k)\subset \left[\xi-b_k,\xi-a_k\right]
\]
and therefore
	\[
	b_k-a_k\geq a.
\]
The other inequality is shown by an explicit calculation. By construction of $\ufl$ it is obvious that for $I \ni i_3:=\min\left\{i\in I|i<i_2\right\}$ one has $v_{i_3}(\xi-a_{k+1})=\ufl(\xi-a_{k+1})$. Define $r:=2\left(x_{i_2}-\left(\xi-b_k\right)\right)$ and $z_0:=x_{i_2}-r$. Then by the periodicity property of $u_{i_0}$ and $u_{i_2}$ it is clear that $z_0$ is a minimum of $u_{i_0}$. Furthermore from the proof of Proposition~\ref{prop:welldefined} one knows that $v_{i_3}$ and $v_{i_2}$ intersect in $\frac{x_{i_3}+x_{i_2}}{2}$. The same argument also shows for the intersection of $u_{i_0}$ and $v_{i_3}$
\begin{align*}
\xi-a_{k+1}=\frac{x_{i_3}+z_0}{2}=\frac{x_{i_3}+x_{i_2}-r}{2}=\frac{x_{i_3}+x_{i_2}}{2}-\frac{r}{2}.
\end{align*}
By the choice of the $Q_k$ it follows $\abs{x_{i_2}-x_{i_3}}=x_{i_2}-x_{i_3}\leq d+2l-2r_1$ and therefore $\frac{\abs{x_{i_2}-x_{i_3}}}{2}\leq d/2+l-r_1<a$. Putting everything together one gets
\begin{align*}
	a_{k+1}-b_k=\left(\xi-b_k\right)-\left(\xi-a_{k+1}\right)=\frac{x_{i_2}+z_0}{2}-\frac{x_{i_3}+z_0}{2}=\frac{x_{i_2}-x_{i_3}}{2}<a
\end{align*}
See Figure~\ref{fig:estimate_periodic} for an illustration.
\begin{figure}
\begin{center}\resizebox{\textwidth}{!}{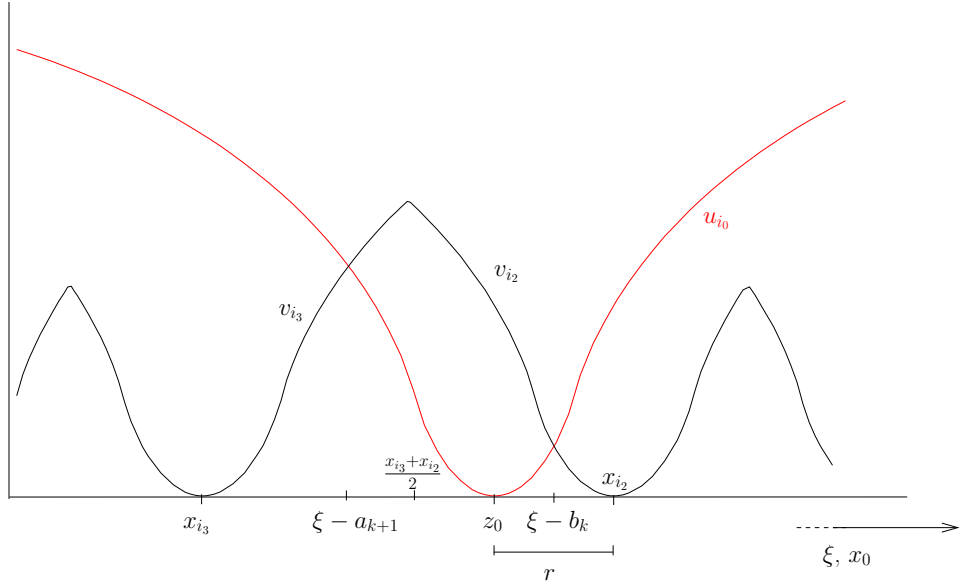}\end{center}
\caption{Illustration to estimate the distances between intersection points of local supersolutions with the periodic supersolution $v_{i_0}$ centered around $x_0$.} \label{fig:estimate_periodic}
\end{figure}

\end{proof}

\begin{lemma}
\label{lem:Aufl_smooth}
Assume that $\xi$ is not a point of discontinuity of the first derivative of $\ufl$. We then have 
$$
\int_\R \frac{\ufl(y) - \ufl(\xi)}{ \abs{y-\xi}^{1+2s}} \,\dy \le \int_\R \frac{ u_{i_0}(y) - u_{i_0}(\xi) }{\abs{y-\xi}^{1+2s}} \,\dy.
$$
\end{lemma}
\begin{proof}
We only consider the part of the integral to the left of $\xi$. With the definition of $\tilde{a}_k, \tilde{b}_k$, the estimate for the other part of the integral follows accordingly. Using the same notation as in Lemma~\ref{lem:distance}, one has $\ufl\leq u_{i_0}$ on $\left[\xi-b_k,\, \xi-a_k\right]$ and so the definition of $\ufl$ together with Lemma~\ref{lem:distance} yields
\begin{equation*}\label{eq:estimate_neg}
\begin{split}
&\quad\int_{\xi-b_k}^{\xi-a_k}{\ufl\left(y\right)-u_{i_0}\left(y\right)}\dy\\
&\leq\min\left\{\int_{\xi-b_k}^{\xi-b_k+a}{u_{i_2}(y)-u_{i_0}(y)}\dy,\int_{\xi-a_k-a}^{\xi-a_k}{u_{i_1}(y)-u_{i_0}(y)}\dy\right\}\\
&\leq \int_{\xi-b_k}^{\xi-b_k+a}{u_{i_2}(y)-u_{i_0}(y)}\dy \leq 0.
\end{split}
\end{equation*}
In the same way using $u_{i_0}\leq \ufl$ on $\left[\xi-a_{k+1},\, \xi-b_k\right]$ one gets
\begin{equation*}\label{eq:estiamte_pos}
\begin{split}
	0&\leq \int_{\xi-a_{k+1}}^{\xi-b_k}{\ufl\left(y\right)-u_{i_0}(y)}\dy\\
	&\leq \min\left\{\int_{\xi-a_{k+1}}^{\xi-a_{k+1}+a}{u_{i_3}(y)-u_{i_0}(y)}\dy, \int_{\xi-b_k-a}^{\xi-b_k}{u_{i_2}(y)-u_{i_0}(y)}\dy\right\}\\
	&\leq \int_{\xi-b_k-a}^{\xi-b_k}{u_{i_2}(y)-u_{i_0}(y)}\dy.
\end{split}
\end{equation*}
Using this one can split up $\int_{\xi-a_{k+1}}^{\xi-a_k}{\frac{\ufl\left(y\right)-u_{i_0}(y)}{\abs{y-\xi}^{1+2s}}}\dy$ in two parts where the integrand is negative and positive, respectively. Noticing that $\frac{1}{\abs{y-\xi}^{1+2s}}\leq \frac{1}{b_k^{1+2s}}$ on $\left[\xi-a_{k+1}, \xi-b_k\right]$ and $\frac{1}{b_k^{1+2s}}\leq \frac{1}{\abs{y-\xi}^{1+2s}}$ on $\left[\xi-b_k, \xi-a_k\right]$ one gets
\begin{align*}	\int_{\xi-a_{k+1}}^{\xi-a_k}{\frac{\ufl\left(y\right)-u_{i_0}(y)}{\abs{y-\xi}^{1+2s}}}\dy&\leq\frac{1}{b_k^{1+2s}}\int_{\xi-a_{k+1}}^{\xi-b_k}{\ufl\left(y\right)-u_{i_0}(y)}\dy\\
	&+\frac{1}{b_k^{1+2s}}\int_{\xi-b_k}^{\xi-a_k}{\ufl\left(y\right)-u_{i_0}(y)}\dy\\
	&\leq\frac{1}{b_k^{1+2s}}\int_{\xi-b_k-a}^{\xi-b_k+a}{u_{i_2}(y)-u_{i_0}(y)}\dy=0,
\end{align*}
where in the last step the peridicity of $u_{i_0}$ and $u_{i_2}$ was used. See Figure~\ref{fig:est_areas} for an illustration.
Inserting a zero in the form $-u_{i_0}\left(\xi\right)+u_{i_0}\left(\xi\right)$ one gets
	\[
	0\geq \int_{\xi-a_{k+1}}^{\xi-a_k}{\frac{\ufl\left(y\right)-u_{i_0}(\xi)}{\abs{y-\xi}^{1+2s}}}\dy-\int_{\xi-a_{k+1}}^{\xi-a_k}{\frac{u_{i_0}(y)-u_{i_0}(\xi)}{\abs{y-\xi}^{1+2s}}}\dy.
\]
Using now $u_{i_0}\left(\xi\right)=\ufl\left(\xi\right)$ and summing up for all $k$ in $\N$ it follows that
	\[
	\int_{-\infty}^{\xi-a_1}\frac{\ufl\left(y\right)-\ufl\left(\xi\right)}{\abs{y-\xi}^{1+2s}}\dy\leq \int_{-\infty}^{\xi-a_1}\frac{u_{i_0}\left(y\right)-u_{i_0}\left(\xi\right)}{\abs{y-\xi}^{1+2s}}\dy.
\]
Furthermore an analogous calculation shows that
	\[
	\int_{\xi+\widetilde{a}_1}^{\infty}\frac{\ufl\left(y\right)-\ufl\left(\xi\right)}{\abs{y-\xi}^{1+2s}}\dy\leq \int_{\xi+\widetilde{a}_1}^{\infty}\frac{u_{i_0}\left(y\right)-u_{i_0}\left(\xi\right)}{\abs{y-\xi}^{1+2s}}\dy,
\]
which together with $\ufl= u_{i_0}$ on $\left(\xi-a_1,\xi+\tilde{a}_1\right)$ yields

	\[
	\int_{\mathbb R}\frac{\ufl\left(y\right)-\ufl\left(\xi\right)}{\abs{y-\xi}^{1+2s}}\dy\leq \int_{\mathbb R}\frac{u_{i_0}\left(y\right)-u_{i_0}\left(\xi\right)}{\abs{y-\xi}^{1+2s}}\dy.
\]

\begin{figure}
\begin{center}\resizebox{\textwidth}{!}{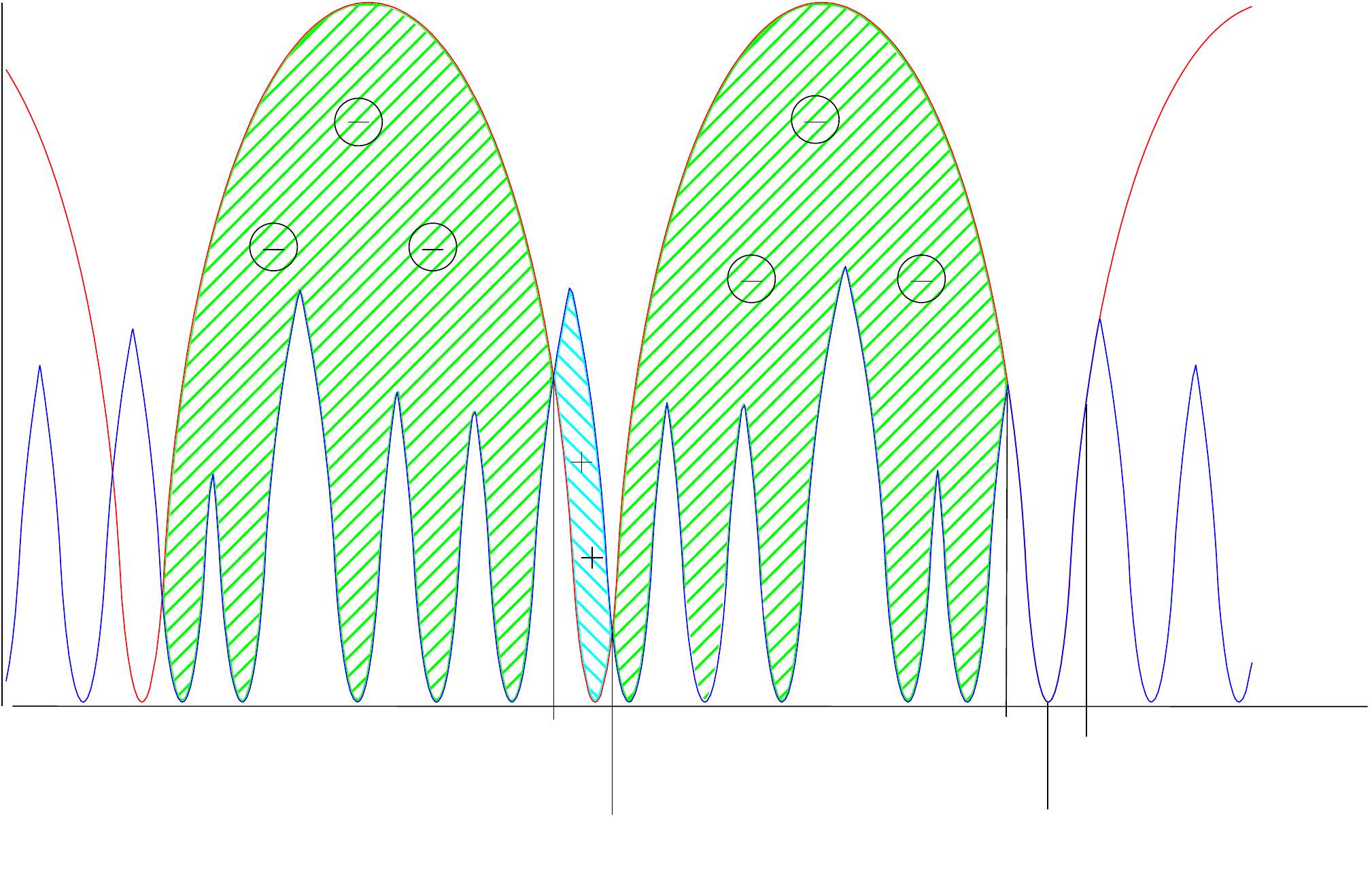}\end{center}
\caption{} \label{fig:est_areas}
\end{figure}

\end{proof}

Next we show that the function $\ufl$ is a weak supersolution to the above mentioned modified problem.
\begin{lemma}
Let $0<\epsilon<\frac{r_0}{4}$, $b+\delta<r_0-2\epsilon$, $F_2<q$ and let $\gfl \colon \R \to [0,\infty)$ be given as
$$
\gfl(x) := \left\{ \begin{array}{ll} q & \textrm{for $x\in \bigcup_{i\in I} [x_i - r_0 + \frac{3}{2}\epsilon, x_i + r_0 - \frac{3}{2}\epsilon]$} \\[1mm]
0 &  \textrm{for $x\notin \bigcup_{i\in I} (x_i - r_0 + \frac{3}{2}\epsilon, x_i + r_0 - \frac{3}{2}\epsilon)$.} \end{array} \right.
$$
Then for all $\psi \in C_c^\infty(\R), \psi \ge 0$ we have 
$$
\left< \ufl, A\psi \right> + \left< -\gfl +F , \psi \right>:= \int_\R \ufl\, (A\psi) \,\dx + \int_\R (-\gfl+F)\, \psi \,\dx \le 0,
$$
for any $F \le \min\{q-F_2, F_1\}$.
\end{lemma}
\begin{proof}
If the support of $\psi$ does not contain any points of discontinuity of $\ufl'$, the statement is clear from Lemma~\ref{lem:Aufl_smooth} by noting that the integral representation of $A\ufl$ is well defined and finite on the whole support of $\psi$ and one can thus apply the fractional Laplacian directly to $\ufl$. A calculation then yields 
\begin{align*}
&	\int_{\mathbb R}\frac{\ufl\left(y\right)-\ufl\left(x\right)}{\left|y-x\right|^{1+2s}}\, \dy-\gfl \left(x\right) \\ 
&	\quad\leq \int_{\mathbb R}\frac{u_{i\left(x\right)}\left(y\right)-u_{i\left(x\right)}\left(x\right)}{\abs{y-x}^{1+2s}}\, \dy-\gfl\left(x\right) \\
&	\quad\leq	\begin{cases}
	F_2-\gfl \left(x\right) & \text{for } x\in \left[x_{i\left(x\right)}-b-\delta,\, x_{i\left(x\right)}+b+\delta\right]\\
	-F_1-\gfl \left(x\right)& \text{for } x\in \left[-a,\, a\right]\backslash \left[x_{i\left(x\right)}-b-\delta,\, x_{i\left(x\right)}+b+\delta\right]
	\end{cases} \\
&	\quad\leq -\min\left\{q-F_2,\, F_1\right\}.
\end{align*}
Using a partition of unity, one can isolate any points of discontinuity. It is possible to locally split $\ufl$ into a piecewise affine function with the same jump and a $C^1\cap H^2_{loc}$-function with bounded second derivatives. Noting that the first derivative of $\ufl$ always admits a \emph{negative} jump, the integral operator $A$ applied to those two parts yields a strongly negative term near the jump for the piecewise affine function and a bounded term for the $C^1\cap H^2_{loc}$-function. The statement in the lemma is thus proved for all test functions $\psi$.
\end{proof}

It is now possible to mollify the function $\ufl$ by a standard mollifier of radius $\epsilon$ to obtain a smooth classical supersolution to the modified (flat) problem.
\begin{corollary}\label{cor:est_flat_smooth}
Let $\usm := \eta_\epsilon * \ufl \in C^\infty$ and $\gsm := \eta_\epsilon * \gfl$. We then have
$$
\A\usm - \gsm +F \le 0
$$ 
for all $F \le \min\{q-F_2, F_1\}$.
\end{corollary}

\begin{proposition}\label{Stufenfunktion}
Let $h>0$, $d>0$, $l>0$, $s\in \left[1/2,1\right)$. For $\Lambda: \mathbb Z\rightarrow \mathbb R$ such that $\abs{\Lambda\left(z_1\right)-\Lambda\left(z_2\right)}\leq 2h\abs{z_1-z_2}^{\alpha}$ with $0<\alpha<1$ there exist a smooth function $\ustep:\, \mathbb R\rightarrow \mathbb R $ and constants $C_0$, $C_1$ and $C_2$, that only depend on $s$ and $\alpha$ such that
\begin{enumerate}[\upshape (i)]
	\item $\ustep(x)=\Lambda(k)$ for any $x\in Q_k$, $k\in\mathbb Z$ \label{piecewiseconstant}
	\item $\left\|\partial_{x}^2\ustep\right\|_{\infty}\leq C_0 \frac{h}{d^2}$  \label{sec_der_est}
	\item 
$	\left|\left(-\Delta\right)^{s}\ustep(x)\right|\leq    C_1\frac{\left(d/2+l/2\right)^{2-2s}}{d^2}h+C_2\frac{h}{\left(d/2+l/2\right)^{2s}}$.
\end{enumerate}
\end{proposition}
\begin{proof}
Parts (i)-(ii) are immediate by mollifying a piecewise constant function. Part (iii) can be seen as follows.
Without loss of generality one can assume $x\in \left[-d/2-l/2,d/2+l/2\right]$.\\
Let $\Pi:=\left(-3\left(d/2+l/2\right),\, 3\left(d/2+l/2\right)\right)$, then one has from the assumptions on $\Lambda$ and from~\eqref{piecewiseconstant} for all $y\in \mathbb R\backslash \Pi$

\begin{align}\label{eq:est_step}
\abs{\ustep\left(x\right)-\ustep\left(y\right)}\leq 6h\frac{\abs{x-y}^{\alpha}}{\left(l+d\right)^{\alpha}}.
\end{align}

because
\begin{align*}
\abs{\ustep\left(x\right)-\ustep\left(y\right)}\leq 2h\frac{\abs{x-y}^{\alpha}}{\left(l+d\right)^{\alpha}}+4h
\end{align*}
but as $y\in \mathbb R\backslash \Pi$ one has $\frac{\abs{x-y}^{\alpha}}{\left(l+d\right)^{\alpha}}\geq 1$ and therefore $4h\leq 4h\frac{\abs{x-y}^{\alpha}}{\left(l+d\right)^{\alpha}}$.

Using that $\ustep$ is smooth and grows less than linear, $\left(-\Delta\right)^s\ustep$ can be a represented as
\begin{align*}
\left(-\Delta\right)^s\ustep\left(x\right)=\frac{1}{2}\int_{\mathbb R}{\frac{\ustep\left(x-y\right)+\ustep\left(x+y\right)-2\ustep\left(x\right)}{\left|y\right|^{1+2s}}}\dy
\end{align*}
by applying some standard estimates~\cite{DiNezza:2011va} one gets

\begin{align*}
&\quad\left|\left(-\Delta\right)^{s}\ustep(x)\right|=\frac{1}{2}\left|\int_{\mathbb R}{\frac{\ustep\left(x-y\right)+\ustep\left(x+y\right)-2\ustep\left(x\right)}{\left|y\right|^{1+2s}}}\dy\right|\\
&\leq\frac{\left\|\partial_x^2\ustep\right\|_{\infty}}{2}\int_{\Pi}{\frac{1}{\left|y\right|^{2s-1}}}\dy\\
&+\frac{1}{2}{\int_{\mathbb R\backslash\Pi}{\frac{\left|\ustep\left(x-y\right)-\ustep\left(x\right)\right|+\left|\ustep\left(x+y\right)-\ustep\left(x\right)\right|}{\left|y\right|^{1+2s}}}\dy}.
\end{align*}
Putting in estimate~\eqref{eq:est_step} for $\ustep$ and calculating the resulting integals it follows that
\begin{align*}
&\quad\left|\left(-\Delta\right)^{s}\ustep(x)\right| \\
&\leq \left\|\partial_x^2\ustep\right\|_{\infty}\int_0^{3\left(d/2+l/2\right)}{\frac{1}{y^{2s-1}}}\dy+{\frac{6h}{\left(l+d\right)^{\alpha}}\int_{\mathbb R\backslash\Pi}{\frac{1}{\left|y\right|^{1+2s-\alpha}}}\dy}\\
&\leq \frac{\left\|\partial_x^2\ustep\right\|_{\infty}}{2-2s}\,3^{2-2s}\left(d/2+l/2\right)^{2-2s}+{\frac{12h}{\left(l+d\right)^{\alpha}}\int_{3\left(d/2+l/2\right)}^{\infty}{y^{-1-2s+\alpha}}\dy}\\
&\leq \frac{\left\|\partial_x^2\ustep\right\|_{\infty}}{2-2s}\,3^{2-2s}\left(d/2+l/2\right)^{2-2s}+\frac{12h}{2s-\alpha}3^{2s-\alpha}\frac{\left(d/2+l/2\right)^{-2s+\alpha}}{\left(d+l\right)^{\alpha}}\\
&\leq \frac{3^{2-2s}}{2-2s}C_0\frac{h}{d^2}\left(d/2+l/2\right)^{2-2s}+\frac{12}{2s-\alpha}\frac{3^{2s-\alpha}}{2^{\alpha}}h\left(d/2+l/2\right)^{-2s}
\end{align*}

where in the last step $\left\|\partial_{x}^2\ustep\right\|_{\infty}\leq C_0 \frac{h}{d^2}$ was used. For
	\[
	C_1=\frac{3^{2-2s}}{2-2s}C_0 \quad \text{and} \quad C_2=\frac{12}{2s-\alpha}\frac{3^{2s-\alpha}}{2^{\alpha}}
\]
one obtains the estimate.
\end{proof}

\begin{lemma}\label{scaling_1}
Let $s>1/2$ and take $\Cv:=\frac{2}{\pi^{2s}}\zeta\left(2s\right)$ and $0<\Cb<1$, $\Ca>5$.\\
Take $q>0$ and $V>0$ as in Proposition~\ref{prop:prop}. Choose $0<F_2<q$ and take $F_1>0$ as in Definition~\ref{def:g}. Choose now $l>0$ such that
\begin{align*}
l>\max\left\{\, 4r_1,\, \left(\frac{\left(C_1+C_2\right)V}{r_1\left(q-F_2\right)}\right)^{1/\left(2s\right)},\, \left(12\left(C_1+C_2\right)Vr_0\right)^{1/\left(2s\right)} , \right.\\
 \frac{1+2F_2r_0r_1+12F_2\left(C_1+C_2\right)V\Cv\Ca^{2s}}{F_2r_0},\, \left.\frac{1}{\left(\Cv\left(\frac{2}{3}\right)^{2s-1}F_2\right)^{1/\left(2s-1\right)}}\right\}
\end{align*}.
Take $d=l$, $\frac{3}{2}l=\frac{d}{2}+l\leq a\leq \Ca l$, $b=\frac{1}{6}\frac{ar_0}{\left(\Cv F_2a^{2s}+r_0\right)}$, $0<\delta < \Cb b$
and $h=\frac{V}{l-2r_1}$. Choose $0<\epsilon<\frac{r_0}{4}$ such that $b+\delta<r_0-2\epsilon$ (possible due to item (\ref{it:rho})). Then we have
\begin{enumerate}[\upshape (i)]
	\item $\rho=b+\delta/2<\frac{r_0}{3}<\frac{a}{18}$, \label{it:rho}
	\item $\left(C_1+C_2\right)V\frac{1}{l^{2s}}\frac{1}{l-2r_1}<\frac{q-F_2}{2}$,
	\item $\left(C_1+C_2\right)V\frac{1}{l^{2s}}\frac{1}{l-2r_1}<\frac{1}{12}\frac{r_0}{\left(\Cv\Ca^{2s}F_2l^{2s}+r_0\right)}$,
	\item $\left\|v\right\|_{\infty}\leq \frac{2\left(F_1+F_2\right)}{\pi^{2s}}\zeta\left(2s\right)a^{2s-1}\left(b+\delta/2\right)\\=\Cv\left(F_1+F_2\right)a^{2s-1}\left(b+\delta/2\right)<\frac{r_0}{2}$,
	\item $-F_1\leq -\frac{1}{6}\frac{r_0F_2}{\left(\Cv\Ca^{2s}F_2 l^{2s}+r_0\right)}$.
\end{enumerate}
\end{lemma}

\begin{proof}
\begin{enumerate}[\upshape (i)]
	\item The choice of $b$, $\delta$, $a$ and $l$ yields
\begin{align*}
b+\delta/2 &\leq2b=\frac{1}{3}\frac{ar_0}{\Cv F_2a^{2s}+r_0}<\frac{1}{3}\frac{r_0}{\Cv F_2a^{2s-1}}\\
&\leq \frac{1}{3}\frac{r_0}{\Cv F_2 \left(\frac{2}{3}\right)^{2s-1}l^{2s-1}}\\
& < \frac{r_0}{3} \frac{1}{\Cv F_2 \left(\frac{2}{3}\right)^{2s-1}\left(\frac{1}{\Cv F_2 \left(\frac{2}{3}\right)^{2s-1}}\right)}\\
&=\frac{r_0}{3}
\end{align*}
from the conditions on $r_0$, $r_1$ and the choice of $l$ one also gets
\begin{align*}
\frac{r_0}{3}<\frac{r_1}{3}<\frac{l}{12}\leq \frac{2}{3}\frac{a}{12}=\frac{a}{18}.
\end{align*}
	\item From the conditions on $l$ we get
	\begin{align*}
	l&>\left(\frac{2\left(C_1+C_2\right)V}{2r_1\left(q-F_2\right)}\right)^{\frac{1}{2s}}\\
	l^{2s}&> \frac{2\left(C_1+C_2\right)V}{2r_1\left(q-F_2\right)}\\
	l^{2s}\left(l-2r_1\right)&> 2r_1l^{2s}>\frac{2\left(C_1+C_2\right)V}{q-F_2}
	\end{align*}
	where in the last step $l>4r_1$ was used. This gives
	\begin{align*}
	\frac{q-F_2}{2}> \left(C_1+C_2\right)V\frac{1}{l^{2s}}\frac{1}{l-2r_1}
	\end{align*}
	\item From the condition
	\[
	l>\frac{1+2F_2r_0r_1+12F_2\left(C_1+C_2\right)V\Cv\Ca^{2s}}{F_2r_0},
\]
on $l$ we get
\begin{align}
r_0F_2l-2F_2r_0r_1-12F_2\left(C_1+C_2\right)V\Cv\Ca^{2s}>1. \label{Bedingung s}
\end{align}
By rearranging some terms we get from the condition
$$
\left(12\left(C_1+C_2\right)Vr_0\right)^{\frac{1}{2s}}<l
$$
on $l$ that
\begin{align*}
12\left(C_1+C_2\right)Vr_0&<l^{2s}\\
&\underset{\eqref{Bedingung s}}{<}l^{2s}\left(r_0F_2l-2F_2r_0r_1-12F_2\left(C_1+C_2\right)V\Cv\Ca^{2s}\right)\\
12\left(C_1+C_2\right)Vr_0&< r_0F_2l^{1+2s}-2F_2\left(r_0r_1+6\left(C_1+C_2\right)V\Cv\Ca^{2s}\right)l^{2s}\\
l^{1+2s}r_0F_2-2r_0r_1F_2l^{2s}&>12\left(C_1+C_2\right)V\Cv F_2\Ca^{2s}l^{2s}+12\left(C_1+C_2\right)Vr_0\\
\left(C_1+C_2\right)V\frac{1}{l^{2s}}\frac{1}{l-2r_1}&< \frac{1}{12}\frac{r_0F_2}{\left(\Cv\Ca^{2s}F_2l^{2s}+r_0\right)}
\end{align*}
\item From Definition~\ref{def:g} we know 
	\begin{align*}
	 F_1+F_2= \left(\frac{\rho}{a-\rho}+1\right)F_2=\frac{a}{a-\rho}F_2.
	\end{align*}
Furthermore from Lemma~\ref{lem:stand_sol} one has	
\begin{align*}
	\left\|v\right\|_{\infty}&\leq \frac{2\left(F_1+F_2\right)}{\pi^{2s}}\zeta\left(2s\right)a^{2s-1}\left(b+\delta\right)
\end{align*}
Putting this together and using item (\ref{it:rho}) one gets
\begin{align*}
\left\|v\right\|_{\infty}&\leq \Cv \frac{a}{a-\left(b+\delta/2\right)}F_2a^{2s-1}\left(b+\delta/2\right)\\
	&< \Cv \frac{a}{\frac{17}{18}a}F_2a^{2s-1}\left(b+\delta/2\right)\\
	&< \frac{18}{17}\Cv F_2 a^{2s-1}2b\\
	&< \frac{18}{17}\Cv F_2a^{2s-1}2\frac{1}{6}\frac{ar_0}{\left(\Cv F_2a^{2s}+r_0\right)}\\
	&< \frac{18}{17}\Cv F_2 a^{2s-1}2\frac{ar_0}{6\Cv F_2a^{2s}}\\
	&=\frac{6}{17}r_0\\
	&< \frac{r_0}{2}
\end{align*}
	\item Because $F_1$ is as in Definition~\ref{def:g} one has
	\begin{align*}
	-F_1&=-\frac{\rho}{a-\rho}F_2=-\frac{b+\delta/2}{a-(b+\delta/2)}F_2 \leq-\frac{b}{a-b}F_2<-\frac{b}{a}F_2\\
	&=-\frac{1}{a}\frac{r_0 a}{6\left(\Cv F_2 a^{2s}+r_0\right)}F_2= -\frac{1}{6}\frac{r_0}{\left(\Cv\Ca^{2s}F_2 l^{2s}+r_0\right)}F_2	
	\end{align*}
\end{enumerate}
\end{proof}

Take now $u:=\usm+\ustep$. Choose the parameters as in Lemma~\ref{scaling_1} and
	\[
	F^{*}:=\frac{1}{2}\min\left\{q -F_2,\, \frac{r_0}{6\left(\Cv\Ca^{2s}F_2 l^{2s}+r_0\right)}F_2\right\}.
\]
then one has $u\geq0$ as just seen and we can now give the 

\begin{proof}[Proof of Theorem~\ref{thm:main} for $s>1/2$]

Let the parameters be as in Lemma~\ref{scaling_1} and take $\ustep\left(x_i\right)=y_i$ for all $i\in I$ which, due to Propositions~\ref{prop:prop} and~\ref{Stufenfunktion}, is (almost surely) possible. From the choice of $\gfl$ and $\eta_\epsilon$ we have

\begin{align*}
-\sum_{i\in I}f_i\left(\omega\right)\phi\left(x-x_i,\, \usm\left(x\right)+\ustep\left(x\right)-y_i\right)\leq -\gfl\ast\eta_\epsilon\left(x\right)
\end{align*}

Using this we have for $F<F^{\ast}$
\begin{align*}
&\quad \A u\left(x\right)-f\left(x,\, u\left(x,\, \omega\right),\, \omega\right)+F\\
&\leq \A\usm\left(x\right)-\sum_{i\in I}f_i\left(\omega\right)\phi\left(x-x_i,\, \usm\left(x\right)+\ustep\left(x\right)-y_i\right) \\
&\quad+F+\A \ustep\left(x\right)\\
&\leq \A\usm\left(x\right)-\gfl\ast\eta_\epsilon\left(x\right)+F+\A \ustep\left(x\right).
\end{align*}
With the results of Corollary~\ref{cor:est_flat_smooth} and Propositions~\ref{prop:prop} and~\ref{Stufenfunktion} it follows that
\begin{align*}
&\quad \A u\left(x\right)-f\left(x,\, u\left(x,\, \omega\right),\, \omega\right)+F\\
&\leq -\min\left\{q -F_2,\, F_1\right\}+F+C_1\frac{\left(d/2+l/2\right)^{2-2s}}{d^2}h+C_2\frac{h}{\left(d/2+l/2\right)^{2s}}\\
&\leq -\min\left\{q-F_2,\, F_1\right\}+F+\left(C_1+C_2\right)\frac{h}{l^{2s}},
\end{align*}
where in the last step $d=l$ was used. Applying now the estimates of Lemma~\ref{scaling_1} and $h=\frac{V}{l-2r_1}$ we get
\begin{align*}
&\quad \A u\left(x\right)-f\left(x,\, u\left(x,\, \omega\right),\, \omega\right)+F\\
&\leq -\min\left\{q-F_2,\, \frac{r_0}{6\left(\Cv\Ca^{2s}F_2l^{2s}+r_0\right)}F_2\right\}+F+\left(C_1+C_2\right)V\frac{1}{l^{2s}}\frac{1}{l-2r_1}\\
&< -\frac{1}{2}\min\left\{q-F_2,\, \frac{r_0}{6\left(\Cv\Ca^{2s}F_2s^{2s}+r_0\right)}F_2\right\}+F^{\ast}\\
&=0,
\end{align*}
which finally concludes the proof for $s>1/2$.
\end{proof}

For the case $s=1/2$ some changes in the choice of parameters have to be performed due to the worse $L^{\infty}$ estimate on $v$ in Lemma~\ref{lem:stand_sol} in this case.

\begin{lemma}\label{lem:scaling_2}
Let $\Crho=\frac{1}{2}\sqrt{\frac{\pi r_0^3}{48\ee^2\left(\frac{36F_2}{17\pi}\right)^3\Ca^3}}$ and choose $0<\Cb<1$, $\Ca>5$.\\
Take $q>0$ and $V>0$ as in Proposition~\ref{prop:prop}. Let $0<F_2<q$ and choose $F_1>0$ as in Lemma~\ref{lem:stand_sol}. Choose now $l>0$ such that
\begin{align*}
l>&\max\left\{\left(\frac{\left(C_1+C_2\right)V}{r_1\left(q-F_2\right)}\right),\,\left(\frac{2V\left(C_1+C_2\right)}{F_2\Crho}\right)^2+4r_1, \right.\,\\
& \left.12\frac{\left(C_1+C_2\right)\Ca}{r_0}V+2r_1 \right\}
\end{align*}

Take $d=l$ and $\frac{3}{2}l=\frac{d}{2}+l\leq a\leq \Ca l$, $b=\frac{1}{2}\min\left\{\sqrt{\frac{\pi r_0^3}{48\ee^2\left(\frac{36F_2}{17\pi}\right)^3}}\frac{1}{\sqrt{a}},\, \frac{r_0}{3}\right\}$, $0<\delta < \Cb b$ und $h=\frac{V}{l-2r_1}$. Finally choose $0<\epsilon<\frac{r_0}{4}$ such that $b<r_0-2\epsilon$ (possible due to item (\ref{eq:est_rho_2})). Then we have
\begin{enumerate}[\upshape (i)]
	\item $l>4r_1$,
	\item $\rho=b+\delta/2<b+\delta<\frac{r_0}{3}<\frac{a}{18}$ \label{eq:est_rho_2},
	\item $\left(C_1+C_2\right)V\frac{1}{l^{1+2\epsilon}}\frac{1}{l-2r_1}<\frac{q-F_2}{2}$,
	\item $\left(C_1+C_2\right)V\frac{1}{l^{1+2\epsilon}}\frac{1}{l-2r_1}<\frac{1}{2}F_2\min\left\{\frac{1}{2}\sqrt{\frac{\pi r_0^3}{48 \ee^2 \left(\frac{36F_2}{17\pi}\right)^3\Ca^3}}\frac{1}{l^{3/2}},\, \frac{r_0}{6\Ca l}\right\}$,
	\item $\left\|v\right\|_{\infty}<\frac{r_0}{2}$,
	\item $-F_1\leq -F_2\min\left\{\frac{1}{2}\sqrt{\frac{\pi r_0^3}{48 \ee^2 \left(\frac{36F_2}{17\pi}\right)^3\Ca^3}},\frac{1}{l^{3/2}},\, \frac{r_0}{6\Ca l}\right\}$.
\end{enumerate}
\end{lemma}

\begin{proof}

\begin{enumerate}[\upshape (i)]
	\item This is clear because $l>\left(\frac{2V\left(C_1+C_2\right)}{F_2\Crho}\right)^2+4r_1$ and $\left(\frac{2V\left(C_1+C_2\right)}{F_2\Crho}\right)^2\geq 0$
	\item From the choice of $\delta$ and $b$ it is clear that $b+\delta\leq 2b\leq \frac{r_0}{3}$.	
					From the conditions on $r_0$, $r_1$ and the choice of $l$ we further have
\begin{align*}
\frac{r_0}{3}<\frac{r_1}{3}<\frac{l}{12}\leq \frac{2}{3}\frac{a}{12}=\frac{a}{18}.
\end{align*}
	\item From the condition on $l$ some calculation gives
	\begin{align*}
	l&> \frac{2\left(C_1+C_2\right)V}{2r_1\left(q-F_2\right)}\\
	l\left(l-2r_1\right)&> 2r_1l>\frac{2\left(C_1+C_2\right)V}{q-F_2}\\
	\frac{q-F_2}{2}&> \left(C_1+C_2\right)V\frac{1}{l}\frac{1}{l-2r_1}.
	\end{align*}
	\item From the condition  
	\[
	l>\left(\frac{2V\left(C_1+C_2\right)}{F_2\Crho}\right)^2+4r_1
\]
on $l$ one obtains
\begin{align*}
&\left(\frac{2V\left(C_1+C_2\right)}{F_2\Crho}\right)^2 < l-4r_1<l-4r_1+\frac{4r_1}{l} \\
&\quad=\frac{l^2-4r_1l+4r_1^2}{l}=\frac{\left(l-2r_1\right)^2}{l}.
\end{align*}
Therefore taking the square root and expanding by $l$ one gets
\begin{align*}
\frac{2V\left(C_1+C_2\right)}{F_2\Crho}<\frac{l-2r_1}{l^{1/2}}=\frac{l\left(l-2r_1\right)}{l^{3/2}},
\end{align*}
which finally, after rearranging and putting in $\Crho$, yields
\begin{align*}
\left(C_1+C_2\right)V\frac{1}{l}\frac{1}{l-2r_1}<\frac{1}{2}F_2\Crho\frac{1}{l^{3/2}}=\frac{1}{2}F_2\frac{1}{2}\sqrt{\frac{\pi r_0^3}{48\ee^2\left(\frac{36F_2}{17\pi}\right)^3\Ca^3}}\frac{1}{l^{3/2}}.
\end{align*}
The second part simply follows by rearranging the condition
\begin{align*}
l>12\frac{\left(C_1+C_2\right)\Ca}{r_0}V+2r_1
\end{align*}
on $l$ such that
\begin{align*}
l-2r_1>12\frac{\left(C_1+C_2\right)\Ca}{r_0}V
\end{align*}
and finally
\begin{align*}
\left(C_1+C_2\right)V\frac{1}{l}\frac{1}{l-2r_1}<\frac{1}{2}\frac{r_0}{6\Ca}\frac{1}{l}.
\end{align*}
	\item From Lemma~\ref{lem:stand_sol} we have
	\begin{align*}
	\left\|v\right\|_{\infty}\leq 2\frac{F_1+F_2}{\pi}\rho\left(2+\log\left(a\right)-\log\left(\pi\rho\right)\right).
	\end{align*}
	The choice of $F_1$ from Definition~\ref{def:g} together with item (\ref{eq:est_rho_2}) yields
	\begin{align*}
	F_1+F_2=\left(\frac{\rho}{a-\rho}+1\right)F_2=\frac{a}{a-\rho}F_2\leq \frac{a}{\frac{17}{18}a}F_2=\frac{18}{17}F_2.
	\end{align*}
	Taking both together one gets
	\begin{align*}
	\left\|v\right\|_{\infty}\leq\left(\frac{36F_2}{17\pi}\right)\rho\left(2+\log\left(a\right)-\log\left(\pi\rho\right)\right).
	\end{align*}
	From the choice of $b$ and $\delta$ we further get
	\begin{align*}
	\rho\leq 2b\leq \min\left\{\sqrt{\frac{\pi r_0^3}{48 \ee^2 \left(\frac{36F_2}{17\pi}\right)^3}}\frac{1}{\sqrt{a}},\, \frac{r_0}{3}\right\}
	\end{align*}
	and therefore by squaring
	\begin{align*}
	\rho^2\leq \min\left\{\frac{\pi r_0^3}{48\ee^2\left(\frac{36F_2}{17\pi}\right)^3}\frac{1}{a},\, \frac{r_0^2}{9}\right\},
	\end{align*}
	so in particular
	\begin{align*}
	\rho^2\leq\frac{\pi r_0^3}{48\ee^2\left(\frac{36F_2}{17\pi}\right)^3}\frac{1}{a}.
	\end{align*}
	Using this by rearranging and adding additional terms we get
	\begin{align*}
	&\quad\frac{a\ee^2}{\pi}\leq \frac{1}{48}\frac{r_0^3}{\left(\frac{36F_2}{17\pi}\right)^3\rho^2} \\
	&<\rho\left(1+\frac{r_0}{2\left(\frac{36F_2}{17\pi}\right)\rho}+\frac{1}{2}\left(\frac{r_0}{2\left(\frac{36F_2}{17\pi}\right)\rho}\right)^2+\frac{1}{6}\left(\frac{r_0}{2\left(\frac{36F_2}{17\pi}\right)\rho}\right)^3\right)\\
	&<\rho\exp\left(\frac{r_0}{2\left(\frac{36F_2}{17\pi}\right)\rho}\right).
	\end{align*}
	where in the last step the standard estimate for the exponential function was applied. This shows that we have
	\begin{align*}
	&a\ee^2\quad<\pi\rho\exp\left(\frac{r_0}{2\left(\frac{36F_2}{17\pi}\right)\rho}\right).
	\end{align*}
	Applying the logarithm gives
	\begin{align*}
	2+\log\left(a\right)<\log\left(\pi\rho\right)+\frac{r_0}{2\left(\frac{36F_2}{17\pi}\right)\rho}
	\end{align*}
	and this finally yields
	\begin{align*}
	\left(\frac{36F_2}{17\pi}\right)\rho\left(2+\log\left(a\right)-\log\left(\pi\rho\right)\right)<\frac{r_2}{2}.
	\end{align*}
	\item The choice of $F_1$ from Definiton~\ref{def:g} and the definition of $\rho$ give
	\begin{align*}
	-F_1&=-\frac{\rho}{a-\rho}F_2<-\frac{b}{a}F_2=-\frac{1}{2}\frac{F_2}{a}\min\left\{\sqrt{\frac{\pi r_0^3}{48\ee^2\left(\frac{36F_2}{17\pi}\right)^3}}\frac{1}{\sqrt{a}},\, \frac{r_0}{3}\right\}\\
	&= -\frac{F_2}{2}\min\left\{\sqrt{\frac{\pi r_0^3}{48\ee^2\left(\frac{36F_2}{17\pi}\right)^3}}\frac{1}{a^{3/2}},\, \frac{r_0}{3a}\right\}.
	\end{align*}
	As $a\leq \Ca l$ we have $-\frac{1}{a}\leq -\frac{1}{\Ca l}$ and therefore
	\begin{align*}
	-F_1&<-\frac{1}{2}\frac{F_2}{a}\min\left\{\sqrt{\frac{\pi r_0^3}{48\ee^2\left(\frac{36F_2}{17\pi}\right)^3}}\frac{1}{\sqrt{a}},\, \frac{r_0}{3}\right\}\\
	&\leq -\frac{F_2}{2}\min\left\{\sqrt{\frac{\pi r_0^3}{48\ee^2\left(\frac{36F_2}{17\pi}\right)^3\Ca^3}}\frac{1}{l^{3/2}},\, \frac{r_0}{3 \Ca l}\right\}.
	\end{align*}	
	\end{enumerate}
\end{proof}

\begin{proof}[Proof of Theorem~\ref{thm:main} for $s=1/2$]
Choose the parameters as in Lemma~\ref{lem:scaling_2} and define $u:=\usm+\ustep$. Then Lemma~\ref{lem:scaling_2} shows $u\geq0$  and for
\begin{align*}
F^{*}=\frac{1}{2}\min\left\{q -F_2,\, F_2\min\left\{\frac{1}{2}\sqrt{\frac{\pi r_0^3}{48 \ee^2 \left(\frac{36F_2}{17\pi}\right)^3\Ca^3}}\frac{1}{l^{3/2}},\, \frac{r_0}{6\Ca l}\right\}\right\}
\end{align*}
using the estimates of Corallary~\ref{cor:est_flat_smooth}, Proposition~\ref{Stufenfunktion} and Lemma~\ref{lem:scaling_2} we can show analogously to the proof of the case $s>1/2$ that
\begin{align*}
\A u\left(x\right)-f\left(x,\, u\left(x,\, \omega\right),\, \omega\right)+F\leq 0
\end{align*}
for all $F<F^{*}$.
\end{proof}

\section{Conclusions}
\label{sec:conclusions}

In this article we have shown existence of a non-trivial pinning threshold for interfaces in elastic media with local obstacles. Models of the kind discussed frequently arise in physics, for example in the propagation of crack fronts in heterogeneous media. Assuming free propagation of such an interface for large enough driving force (which is trivial to obtain under some conditions on the heterogeneity), we have shown the transition of a microscopically \emph{viscous kinetic relation} (force=velocity) for interfaces in elastic media with random obstacles to a \emph{stick-slip behavior} on larger scales. The construction of the supersolution has been constrained to the 1+1 dimensional case, i.e., that of a 1-dimensional interface propagating in a 2-dimensional plane. In many cases, this is the physically relevant situation. The $n$-dimensional case is still open due to technical difficulties concerning mostly the compensation of errors arising from modifying a periodic solution.

Furthermore, we have shown a percolation result, namely a non-trivial percolation threshold for the existence of an infinite cluster in next-nearest neighbor site percolation that is the graph of an only logarithmically growing function. 

\section*{Acknowledgement}
PWD and MS gratefully acknowledge support from the DFG research unit FOR 718 `Analysis and Stochastics in Complex Physical Systems'.

\bibliographystyle{alpha}
\bibliography{elastic}

\end{document}